\documentclass{amsart}
\usepackage{graphicx} 

\usepackage{amssymb,amsthm}
\usepackage{enumitem}
\usepackage{mathtools, hyperref,doi,url,hyperref}

\usepackage{crossreftools}
\usepackage[mathscr]{euscript}
\usepackage{xcolor}

\newtheorem{thm}{Theorem}[section]

\newtheorem{cor}[thm]{Corollary}
\newtheorem{lem}[thm]{Lemma}
\newtheorem{prop}[thm]{Proposition}
\theoremstyle{remark}

\newcommand{\D}{\mathbb{D}}
\newcommand{\R}{\mathbb{R}}

\newcommand{\HOLO}{\mathcal H}
\newcommand*\CONJ[1]{\overline{#1}}

\newcommand{\norm}[1]{\left\Vert#1\right\Vert}
\newcommand{\abs}[1]{\left\vert#1\right\vert}

\newcommand*\BDM{\textrm{BMD}}

\newcommand\HL{ \mathcal{HL}} 

\begin{document}
\allowdisplaybreaks

\title[The edge case of the Hilbert matrix operator between weighted Bergman spaces]{The edge case of the Hilbert matrix operator between weighted Bergman spaces}

\author{David Norrbo}

\address{Department of Mathematical Sciences, Faculty of Information Technology and Electrical Engineering, NTNU – Norwegian University of Science and Technology, Trondheim, Norway}
\email{david.norrbo@ntnu.no}

\thanks{}




\begin{abstract}
The Hilbert matrix operator $H$ is bounded on the standard weighted Bergman spaces $A^p_\alpha,\, p>\alpha+2>1$. We examine the limiting case when $\alpha=p-2$, and for $p\leq 2$, we characterize the logarithmic modification of the standard weight needed for the domain $A^p_{v_{\gamma,p}}$ and codomain $A^p_{v_{\beta,p}}$ to ensure boundedness. In particular, we prove that $H:A^2_{\log^\gamma}  \to A^2_{\log^\beta}$ is bounded if and only if $\gamma-\beta\geq 2$ and $\gamma>1$, solving an open problem posed by B. Karapetrovi\'c and M. Jevti\'c. We also settle the situation for $p>2$ in all cases, except when the logarithmic exponents $\gamma$ and $\beta$ satisfy $\gamma-\beta=p$. 

Moreover, we completely solve the same problem when the domain and codomain are certain weighted reflexive $\ell^p$ spaces. All results are extended to the generalized Hilbert operator $H_g$.

\end{abstract}

\maketitle

\section{Introduction}
The Hilbert matrix acting as an operator $H$ on classical analytic function spaces via the Taylor coefficients $(a_n)\mapsto \sum_{n=0}^{\infty} a_n/(m+n+1)$ has been an active research area since the year 2000, when Diamantopoulos and Siskakis \cite{Diamantopoulos2000} proved that $H$ is bounded on the reflexive Hardy spaces $H^p,\, 1<p<\infty$ (see also \cite{Cima_Stegenga_1989}), $H(H^{\infty})\not\subset H^{\infty}$ and $H(H^1)\not\subset H^1$.

In 2004, Diamantopoulos \cite{Diamantopoulos2004} proved that $H$ is bounded on the standard (unweighted) Bergman spaces $A^p, \, 2<p<\infty$. Looking at the growth restriction of functions in $H^p$ and $A^p$, we can see that boundedness seems to break down if we cannot choose $\eta<1$ such that $\sup_{z\in\D}(1-\abs{z})^{\eta}\norm{\delta_z}<\infty$, where $\delta_z(f) = f(z),\, \in \D$. The same holds true for the standard weighted Bergman spaces, which follows \cite[Theorem 1.1]{Jevtic2017} by Jevti\'c and Karapetrovi\'c, that is, $H$ is bounded on $A^p_\alpha$ if and only if $p>2+\alpha>1$, where the growth of the norm of the evaluation map is given by $\eta=(2+\alpha)/p$. Furthermore, in \cite[p.~11]{Lanucha2012}, it was proved that $H$ does not map the Bloch space into the Bloch space. We have a stronger result by Jevti\'c and Karapetrovi\'c, $H$ does not map \textrm{VMOA} into the Bloch space, \cite[Corollary 2.3]{jevtic2017hilbert}. So far, the examples mentioned hints at reflexiveness of a space $X$ and suitable growth of $r\mapsto\norm{\delta_r}_{X^*}$ seem to be two important factors in determining whether $H$ will act as a bounded operator on a complex argument invariant space $X$. Even if \textrm{VMOA} is much smaller than the Bloch space, the growth of $r\mapsto \norm{\delta_r}$ is the same, $\norm{\delta_r} \asymp \log \frac{e}{1-r}$. However, if the domain is small enough compared to the codomain, $H$ will act boundedly. Some examples found in \cite{Lanucha2012} are $H^\infty \to \textrm{BMOA}$ (Theorem 1.2),  $\mathcal B_{\log^{1+\epsilon}} \to \mathcal B,\, \epsilon>0$ (Proposition 5.1). They also showed that $H\colon \mathcal B \to \mathcal B_{\log^{-1}}$ (Proposition 5.2), and give an asymmetrical $\log$-$H^1$-space to be mapped into $H^1$ by $H$ (Theorem 2.3). Ye and Sheng recently proved that $H$ maps $H^1$ boundedly into $H^1_{\log^{-1}}$ \cite{ye2025norm}. 

Concerning the Bergman spaces, in \cite[p.~2810]{Dostanic2008} it is proved that the Hilbert (matrix) operator $H$ is not even well defined on the Bergman space $A^2$. However, it is well defined on $A^2_{\log^{\alpha}}, \, \alpha>1$, and it maps $A^2_{\log^{\alpha}}$ into $A^2$ when $\alpha>2$ \cite[Proposition 4.4 and Theorem 4.5]{jevtic2017hilbert}. The open question stated in \cite[Remark 4.6]{jevtic2017hilbert}, which concerns the cases $1<\alpha\leq 2$, seems to still be open. We completely solve this question by proving that $H\colon A^2_{\log^{\alpha}} \to A^2$ is bounded when $\alpha=2$, but unbounded when $1<\alpha<2$. Our results not only generalize to a more informative variant of the open question, see \cite[p.~568]{karapetrovic2018libera}, but we also consider a reflexive non-Hilbert space counterpart, which we completely solve in the case $p\leq 2$, see Theorem \ref{thm:main}, our main theorem. In order to achieve this result, we first characterize boundedness when the operator acts between a variant of Hardy-Littlewood spaces; see our second main result, Theorem \ref{thm:mainHL}, in which we deal with all real $p>1$. The proof combines some theory on generalized averaging operators and their adjoint operators on $L^p(1,\infty)$ and we derive suitable isomorphisms to connect these operators on $L^p(1,\infty)$ with the Hilbert operator acting on the generalized Hardy-Littlewood spaces; see Lemma \ref{lem:ModAvgBddOnLp}--Corollary \ref{cor:OtherOperatorsHL}. 

The methods used allow the sufficiency condition of boundedness to extend to the generalized Hilbert operators introduced in \cite{Galanopoulos2014}, see Theorem \ref{thm:mainHg}.

Information about the growth of $r\mapsto \norm{\delta_r}$ on the reflexive Hardy and Bergman spaces is gathered in \cite[p.~4363]{Norrbo2026}, while the upper bound on Bloch, \textrm{VMOA} and \textrm{BMOA} can be found in, for example, \cite[p.~430]{anderson1991inner} and \cite[Corollary 5.3]{girela2001analytic}, respectively. The lower bound is obtained using the  functions 
\[
w\mapsto f_z(w) := \log\frac{e}{1-\overline{z}w}, \quad z\in\mathbb D.
\]

Our main focus lies on the following weighted Bergman spaces
\[
A^p_{v_{\gamma,\alpha}}:= \bigg\{f\in \HOLO(\D) : \norm{f}_{A^p_{v_{\gamma,\alpha}}}^p:=  \int_{\D} \abs{f(z)}^p   v_{\gamma,\alpha}(z)\, dA(z) \bigg\}, \ \alpha>-1, \gamma\in \mathbb R,
\]
where $dA(re^{i\theta}) = r\,dr\,d\theta /\pi$ and 
\[
v_{\gamma,\alpha}(z) = C_{\gamma,\alpha}\log\bigg(\frac{e}{1-\abs{z}^2} \bigg)^{\gamma} (1-\abs{z}^2)^{\alpha};
\]
$C_{\gamma,\alpha}$ is a normalizing constant making $f\equiv 1$ have unit norm. 

In the Introduction, we used the following standard conventions: 
\[
A^p = A^p_{v_{0,0}} ,\quad A^p_{\alpha} =A^p_{v_{0,\alpha}} \quad \text{and} \quad A^p_{\log^{\gamma}} = A^p_{v_{\gamma,0}}.
\]

We are also interested in the generalized Hardy-Littlewood spaces,

\begin{equation}\label{eq:DefOfHL}
\HL^p_{\gamma,\alpha} := \Big\{\HOLO(\D)\ni f\colon z\mapsto\sum_{n=0}^{\infty} a_n z^n  :  \norm{f}_{\HL^p_{\gamma,\alpha}}^p  :=  \sum_{n=0}^{\infty} \abs{a_n}^p (n+3)^{p-\alpha-3} (\log(n+3))^{\gamma}  \Big\},
\end{equation}
where $\HL^p_{0,-1}$ is the classical Hardy-Littlewood spaces.

The Hilbert (matrix) operator $H_I$ is formally defined on $\HOLO(\D)$ as
\[
H_I f(z) = \int_0^1 \frac{ f(t) }{ 1-tz } \, dt ,\quad  z\in \D.
\]

It is bounded on $A^p_{v_{0,\alpha}}$ if and only if $p-2>\alpha>-1$ \cite[Theorem 1.1]{Jevtic2017}. Our main result examines the limiting case when $\alpha=p-2$ in which case we must add different logarithmic factors to the weights to make the operator bounded. 

\begin{thm}\label{thm:main}
Let $p\in(1,\infty)$. The Hilbert operator $H_I$ is well defined on $A^p_{v_{\gamma,p-2}}$ if and only if $\gamma>p-1$.

Moreover, 
\begin{enumerate}[label={(\Roman*)}]
\item for $p\leq 2$ and $\gamma>p-1$,
\[
H_I\colon  A^p_{v_{\gamma,p-2}} \to A^p_{v_{\beta,p-2}}
\]
is bounded if and only if $\gamma-\beta\geq p$;

\item for $p > 2$ and $\gamma>p-1$,
\[
H_I\colon  A^p_{v_{\gamma,p-2}} \to A^p_{v_{\beta,p-2}}
\]
is bounded if $\gamma-\beta > p$ and unbounded if $\gamma-\beta < p$.
\end{enumerate}
\end{thm}

We conjecture that: 
\[
H_I\colon  A^p_{v_{\gamma,p-2}} \to A^p_{v_{\beta,p-2}}
\]
is bounded when $p > 2$, $\gamma=\beta +p> p-1$.

In the case $p=2$, Theorem completely solves the open question state in \cite[p.~568]{karapetrovic2018libera}, see also \cite[Remark 4.6]{jevtic2017hilbert}.

We also present the following theorem, which completely characterizes the situation on the corresponding generalized Hardy-Littlewood spaces \eqref{eq:DefOfHL}.

\begin{thm}\label{thm:mainHL}
Let $p\in(1,\infty)$. The Hilbert operator $H_I$ is well defined on $\HL^p_{\gamma,p-2}$ if and only if $\gamma>p-1$.

Moreover, for $\gamma>p-1$,
\[
H_I\colon  \HL^p_{\gamma,p-2} \to \HL^p_{\beta,p-2}
\]
is bounded if and only if $\gamma-\beta\geq p$.
\end{thm}

\section{Preliminaries}

We begin to highlight that for a fixed $y\geq 0$ and $x$ in a compact subset of $(0,\infty)$, we have
\[
\log (x+y+1)\asymp \log (x+1) \quad\text{ and }\quad  x+y\asymp x.
\]
These will be used frequently. Here $x\asymp y$ means $x\lesssim y$ and $x\gtrsim y$, which means there are constants $C_1,C_2>0$ such that $x\leq C_1 y$ and $x\geq C_2 y$, respectively. Dependencies will be written as a subscript. We also use $f(n) \sim g(n)$ as $n\to\infty$ if $\lim_{n\to\infty} f(n)/g(n) = 1$. The H\"older conjugate of $p>1$ will be denoted by $p'$, that is, $1/p+1/p'=1$. We also mention that $p>1$ means $1<p<\infty$.

We formally define the (original) Hilbert matrix operator $H_M$ on $\HOLO(\D)$ as
\[
H_M f(z) = \sum_{m=0}^{\infty}\sum_{n=0}^{\infty}  \frac{ a_n }{m+n+1} z^m, \quad z\in \D, 
\]
where $f \colon z\mapsto \sum_{n=0}^{\infty} a_n z^n$.

Considering
\[
\int_0^1 \sum_{n=0}^{\infty} a_n t^n t^m z^m dt, \quad  z\in\D,
\]
it follows by fubini's theorem that given $f\in\HOLO(\D)$, with its Taylor coefficients satisfying
\begin{equation}\label{eq:SuffCondWellDef}
\sum_{n=0}^{\infty}\frac{\abs{a_n}}{n+1}<\infty,
\end{equation}
$H_I f$ and $H_M f$ are well defined, and $H_I f=H_M f$.

We will also make use of the sublinear Hilbert operator \cite[p.~235]{Galanopoulos2014}, a nonlinear modulus Hilbert operator, which is formally defined as
\[
\widetilde H_I f(z) =  \int_0^1 \frac{ \abs{ f(t) } }{ 1-tz } \, dt.
\]

Moreover, $H_I$ and $\widetilde H_I$ are well defined if the weaker condition $\int_0^1\abs{f(r)} \, dr<\infty$ is satisfied, which the following Lemma shows is satisfied if $f\in A^p_{v_{\gamma,\alpha}}$ when  $\gamma>p-1>0$ and $\alpha = p-2$. 

\begin{lem}\label{lem:MomentsOfAbsH}
Let $1<p<\infty$, $\alpha>-1$ and $\gamma\in\mathbb R$. For $n\in\mathbb N\cup\{0\}$, it holds that
\[
\int_0^1\abs{f(r)} r^n \, dr \lesssim\bigg(\frac{1}{2}\int_0^1  \bigg(\frac{ \log(\frac{e}{1-r} )^{-\gamma}  }{(1-r)^{\alpha+1}}\bigg)^{\frac{p'}{p}} r^n \, dr   \bigg)^{\frac{1}{p'}} \norm{f}_{A^p_{v_{\gamma,\alpha}}},
\]
where the integral is finite if $\gamma>p-1$ and $\alpha = p-2$.
\end{lem}
\begin{proof}
The evaluation map bound on the Hardy spaces \cite[Lemma on p.~36]{Duren1970} applied to $f_r\colon z\mapsto f(rz),\, z\in\D$ yields for all $z\in\D$
\[
\abs{f_r(z)} \lesssim  \bigg(\frac{1}{1-\abs{z}^2} \int_0^{2\pi} \abs{f_r(e^{it})}^p \, dt \bigg)^{\frac{1}{p}}.
\]
By Hölder's inequality, we have for every $n\in\mathbb N\cup\{0\}$,
\begin{equation}\label{eq:HilbertOpWellDefined}
\begin{split}
\int_0^1\abs{f(r)} r^n \, dr &= 2\int_0^1\abs{f_r(r)} r^{2n+1}  \, dr  \\
&\lesssim    \int_0^1  \bigg(\frac{ v_{\gamma,\alpha}(r)^{-1}  }{1-r^2} \bigg)^{\frac{1}{p}} \bigg( \int_0^{2\pi} \abs{f_r(e^{it})}^p \, dt\, v_{\gamma,\alpha}(r)  \bigg)^{\frac{1}{p}}  r^{2n+1} \, dr \\
&\leq   \bigg(\int_0^1  \bigg(\frac{ v_{\gamma,\alpha}(r)^{-1}  }{1-r^2} \bigg)^{\frac{p'}{p}} r^{2n+1} \, dr   \bigg)^{\frac{1}{p'}} \bigg(\int_0^1  \int_0^{2\pi} \abs{f_r(e^{it})}^p \, dt\,  v_{\gamma,\alpha}(r) r \, dr   \bigg)^{\frac{1}{p}}\\
&\lesssim_{\alpha,\gamma}\bigg(\int_0^1  \bigg(\frac{ \log(\frac{e}{1-r} )^{-\gamma}  }{(1-r)^{\alpha+1}}\bigg)^{\frac{p'}{p}} r^n \, dr   \bigg)^{\frac{1}{p'}} \norm{f}_{A^p_{v_{\gamma,\alpha}}}.
\end{split}
\end{equation}
The last integral is finite if 
\[
(\alpha+1)p'/p = (\alpha+1)/(p-1) <1,\quad \text{ or } \quad (\alpha+1)/(p-1) = 1 \text{ and } \gamma p'/p>1.
\]
The first condition is the classical parameter restriction $p>2+\alpha$ \cite[Theorem 1.1]{Jevtic2017},  while the second is the more interesting case, $p=2+\alpha$ and $\gamma>p-1$.

\end{proof}

Similarly, an application of Hölder's inequality shows that for $f\in \HL^p_{\gamma,\alpha}$ with $\alpha=p-2$ and $\gamma>p-1>0$, \eqref{eq:SuffCondWellDef} holds.

An application of the Hardy-Littlewood inequalities yields a relationship between $\HL^p_{\gamma,\alpha}$ and $A^p_{v_{\gamma,\alpha}}$, but to see this we need the following lemma.

\begin{lem}\label{lem:asympOfWeightInSeries}
As $n\to \infty$, we have  
\[
\int_0^1 r^n (1-r)^{\alpha} \bigg(\log\frac{e}{1-r} \bigg)^{\gamma} \, dr \sim\begin{cases}
 \Gamma(1+\alpha) \frac{(\log (n+3))^{\gamma}}{(n+3)^{1+\alpha}},& \text{if }  \gamma\in\R,\ \alpha>-1,\\
\frac{ (\log (n+3))^{1+\gamma}  }{\abs{\gamma}-1}  ,& \text{if } \gamma<-1,\ \alpha=-1.
\end{cases}
\]

\end{lem}
\begin{proof}
For the first case, we rewrite the integral on a form such that we can apply dominated convergence. To this end, we apply the substitution $r\mapsto 1-r/n$, which for $n\in\mathbb N$ gives us

\begin{align*}
\int_0^1 r^n (1-r)^{\alpha} \bigg(\log\frac{e}{1-r} \bigg)^{\gamma} \, dr &=  \frac{  (\log(en))^{\gamma}  }{n^{1+\alpha}} \int_0^n \Big( 1-\frac{r}{n} \Big)^n r^{\alpha} \bigg( \frac{\log\frac{en}{r} }{ \log(en)  }  \bigg)^{\gamma} \, dr \\
&=  \frac{  (\log(en))^{\gamma}  }{n^{1+\alpha}} \int_0^\infty \chi_{(0,n)}(r) \Big( 1-\frac{r}{n} \Big)^n r^{\alpha} \bigg( 1 +\frac{\log\frac{1}{r} }{ \log(en)  }  \bigg)^{\gamma} \, dr.
\end{align*}
If $\gamma \geq 0$, using the inequality $1-x \leq e^{-x}, x>0$, we can see that 
\[
 \chi_{(0,n)}(r) \Big( 1-\frac{r}{n} \Big)^n \bigg( 1 +\frac{\log\frac{1}{r} }{ \log(en)  }  \bigg)^{\gamma} \leq e^{-r} \bigg( 1 +\abs{\log\frac{1}{r} }  \bigg)^{\gamma}.
\]
If $\gamma < 0$, we similarly have
\begin{align*}
 \chi_{(0,n)}(r) \Big( 1-\frac{r}{n} \Big)^n \bigg( 1 +\frac{\log\frac{1}{r} }{ \log(en)  }  \bigg)^{\gamma} &\leq e^{-r}  \Bigg( \chi_{[1,\infty)}(r) \bigg( 1 -\frac{\log r }{ \log(er)  }  \bigg)^{\gamma} +\chi_{(0,1)}(r) \Bigg) \\
 &= e^{-r}  \Bigg( \chi_{[1,\infty)}(r) \bigg( \frac{1 }{ \log(er)  }  \bigg)^{\gamma} +\chi_{(0,1)}(r)\Bigg).
\end{align*}

hence, an application of the dominated convergence theorem yields
\[
\lim_{n\to\infty} \frac{(n+3)^{1+\alpha}}{(\log(n+3))^{\gamma}} \int_0^1 r^n  (1-r)^{\alpha} \bigg(\log\frac{e}{1-r} \bigg)^{\gamma} \, dr = \int_0^{\infty} e^{-r} r^{\alpha}\, dr = \Gamma(1+\alpha).
\]

Now, if $\alpha=-1$ and $\gamma<-1$, let $\epsilon\in(0,1)$. A part of the calculations above yields
\[
\int_0^1 r^n (1-r)^{\alpha} \bigg(\log\frac{e}{1-r} \bigg)^{\gamma} \, dr = \bigg( \int_0^1 +\int_1^n \bigg) \Big( 1-\frac{r}{n} \Big)^n \Big( \log\frac{en}{r} \Big)^{\gamma} \, \frac{dr}{r}. 
\]
For the second integral, we obtained by the above calculations
\[
\int_{\epsilon}^n \Big( 1-\frac{r}{n} \Big)^n \Big( \log\frac{en}{r} \Big)^{\gamma} \, \frac{dr}{r} \asymp_{\epsilon} (\log (n+3))^{\gamma}.
\]
For the first integral, there exists for every $n\in\mathbb N$ an $r(n)\in [0,\epsilon]$ such that
\[
\int_0^{\epsilon}  \Big( 1-\frac{r}{n} \Big)^n \Big( \log\frac{en}{r} \Big)^{\gamma} \, \frac{dr}{r} = \Big( 1-\frac{r(n)}{n} \Big)^n \int_0^{\epsilon} \Big( \log\frac{en}{r} \Big)^{\gamma} \, \frac{dr}{r}.
\]
Furthermore,
\begin{align*}
\int_0^{\epsilon} \Big( \log\frac{en}{r} \Big)^{\gamma} \, \frac{dr}{r} =\int_0^{\frac{\epsilon}{n}}\Big( \log\frac{e}{r} \Big)^{\gamma} \, \frac{dr}{r} =\int_{\frac{en}{\epsilon}}^{\infty} ( \log(r) )^{\gamma} \, \frac{dr}{r}  = \frac{1}{\abs{\gamma}-1} (\log (en/\epsilon))^{1+\gamma} \sim \frac{1}{\abs{\gamma}-1} (\log (n+3))^{1+\gamma} 
\end{align*}
as $n\to\infty$, where the last equality follows from integration by parts (or the substitution $r\mapsto e^r$), 
\begin{equation}\label{eq:IntByParts}
\int_x^\infty  \frac{dt}{t (\log t)^{\beta}} = \frac{\log t}{ (\log t)^{\beta}}\bigg|_x^\infty  + \beta \int_x^\infty  \frac{dt}{t (\log t)^{\beta}}, \quad \beta >1.
\end{equation}
We have now obtained 
\[
\lim_{n\to\infty} \frac{   \int_0^1 r^n (1-r)^{\alpha} \bigg(\log\frac{e}{1-r} \bigg)^{\gamma} \, dr        }{   \Big( 1-\frac{r(n)}{n} \Big)^n \frac{ (\log (n+3))^{1+\gamma}  }{\abs{\gamma}-1}   } = 1.
\]
To obtain the statement in the Lemma, we note that 
\[
e^{-\epsilon} \leq \liminf_{n\to\infty}\Big( 1-\frac{r(n)}{n} \Big)^n \leq \limsup_{n\to\infty}\Big( 1-\frac{r(n)}{n} \Big)^n \leq 1,
\]
hence, let $\epsilon \to 0$ and we are done. 
\end{proof}

Combined with Lemma \ref{lem:asympOfWeightInSeries}, the Hardy-Littlewood inequalities \cite[Theorems 6.2 and 6.3]{Duren1970}, applied to the dilations $f_r$ of a function $f\in \HOLO(\D)$, gives us the following lemma. 
\begin{lem}\label{lem:HLineq}
Let $\alpha>-1$, $\gamma\in\mathbb R$, and $f\in\HOLO(\D)$. 
\begin{enumerate}[label={(\alph*)}]
\item If $1<p\leq 2$, then
\[
\norm{f}_{\HL^p_{\gamma,\alpha}}\lesssim_{p,\gamma,\alpha} \norm{f}_{A^p_{v_{\gamma,\alpha}}};
\]
\item if $p\geq 2$, then
\[
\norm{f}_{A^p_{v_{\gamma,\alpha}}} \lesssim_{p,\gamma,\alpha} \norm{f}_{\HL^p_{\gamma,\alpha}}.
\]
\end{enumerate}
Moreover, if we restrict ourselves to functions $f\colon z\mapsto\sum_{n=0}^{\infty}a_n z^n$ with $a_n \geq 0, \, a_n\geq a_{n+1},\, n\in\mathbb N $ and $a_n\to 0$ as $n\to\infty$, we have
\begin{enumerate}[resume,label={(\alph*)}]
\item If $p>1$, then
\[
\norm{f}_{\HL^p_{\gamma,\alpha}}\asymp_{p,\gamma,\alpha} \norm{f}_{A^p_{v_{\gamma,\alpha}}}.
\]
\end{enumerate}

\end{lem}
The last statement follows from a comparability result for $H^p$ (instead of only the Hardy-Littlewood inequalities), which can be found in, for example, \cite[Proof of Theorem A]{Pavlovic2013}.

One immediate consequence of Lemma \ref{lem:HLineq} is that  \eqref{eq:SuffCondWellDef} holds for $f\in A^p_{v_{\gamma,p-2}}$, $1<p\leq 2,\, \gamma>p-1$.

\section{The Hilbert operator between generalized Hardy-Littlewood spaces}

We begin by presenting a necessary condition for $H_I$ to be bounded on $\HL^p_{\gamma,p-2}, \, p>1,\, \gamma\in\mathbb R$. Notice that  
\begin{equation}\label{eq:CoeffHilbAbs}
\Big[z\mapsto \sum_{n=0}^{\infty} a_n z^n \Big] \in \HL^p_{\gamma,p-2} \quad \Longleftrightarrow \quad \Big[z\mapsto \sum_{n=0}^{\infty} \abs{a_n} z^n \Big] \in \HL^p_{\gamma,p-2},
\end{equation}
the two functions have equal norms, and that $H_If = H_Mf$ if $a_n\geq 0,\, n\in\mathbb N$. Since $H_M$ has a nonnegative kernel, it follows that for $\gamma,\beta\in\mathbb R$, $H_I\colon \HL^p_{\gamma,p-2} \to \HL^p_{\beta,p-2}$ is bounded/well defined/unbounded if and only if $H_M\colon \HL^p_{\gamma,p-2} \to \HL^p_{\beta,p-2}$ is.

We proceed with obtaining necessary conditions for $H_I=H_M$ being well defined, as well as, bounded.
\begin{prop}\label{prop:necessaryCondHL}
Let $p>1$. The operator $H_I$ is well defined on $\HL^p_{\gamma,p-2}$ if and only if $\gamma>p-1$. Moreover,  if $\gamma>p-1$ and $\gamma-\beta <p$, then $H_I(\HL^p_{\gamma,p-2})\not\subset \HL^p_{\beta,p-2}$.
\end{prop}

\begin{proof}
Assume first $\gamma\leq p-1$. Then $\HL^p_{p-1,p-2} \subset \HL^p_{\gamma,p-2}$, so it suffices to consider $\HL^p_{p-1,p-2}$. Let $f_p$ be the analytic function with Taylor coefficients
\[
a_n = \frac{1}{   \log(n+3)  \log(\log(e(n+3)))   }, \quad n\in\mathbb N.
\]
Then
\[
\norm{f}_{\HL^p_{p-1,p-2}}^p = \sum_{n=0}^{\infty}  \frac{1}{(n+3)  \log(n+3)  (\log\log(e(n+3)))^p } < \infty
\]
so $f\in \HL^p_{p-1,p-2}$. The fact that the sum is finite follows from comparing with integrals and integrating by parts; the sum can be compared with
\[
\int_{e^e}^{\infty}  \frac{dx}{ x (\log x) (\log\log x)^p } = \bigg[ \frac{\log\log x }{  (\log\log x)^p  } \bigg]_{e^e}^{\infty} +p \int_{e^e}^{\infty}  \frac{dx}{\log x (\log x) (\log\log x)^p },\quad p>1.
\]

Next, we prove that $H_M f$ is not well defined. This follows immediately again from a comparison with integrals. Indeed,
\[
H_M f(2) =  \sum_{n=0} \frac{1}{ (n+3) \log(n+3)  \log(\log(e(n+3)))   }, 
\]
and
\[
\int_{e^e}^\infty \frac{dt}{ t (\log t) (\log(\log t) )  }  = \lim_{t\to +\infty}\log(\log(\log t)) = \infty.
\]

Assume, hereafter, that $\gamma>p-1>0$. As already mentioned, a straightforward application of H\"older's inequality proves that \eqref{eq:SuffCondWellDef} holds, hence, $H_I=H_M$ is well defined on $\HL^p_{\gamma,p-2}$.

For the second statement, assume additionally that $\gamma-\beta<p$ and let $f\in\HOLO(\D)$ have the Taylor coefficients $a_n= (\log (n+3))^{-(\beta+p+1)/p}$. We have
\[
\norm{f}_{ \HL^p_{\gamma,p-2}  }^p = \sum_{n=0}^{\infty}  \frac{  (\log (n+3))^{\gamma-(\beta+p+1)}   }{n+3} 
\]
so $f\in  \HL^p_{\gamma,p-2}$, because $\gamma-\beta-p-1<-1$. However, since $(\beta+p+1)/p>1$, we have
\[
H_M f(m) = \sum_{n=0}^{\infty}  \frac{a_n}{m+n+1} \gtrsim  \sum_{n=m}^{\infty}  \frac{a_n}{n+3} \gtrsim_{\beta,p}  (\log (m+3))^{-(\beta+1)/p},
\]
where the second $\gtrsim$ can be obtained by a comparison with integrals since the summand is decreasing. The integral $\int_x^\infty  \frac{dt}{t (\log t)^{\beta}}, \, \beta>1$ can be evaluated by integration by parts, see \eqref{eq:IntByParts}.

Continuing, 
\[
\norm{f}_{\HL^p_{\beta,p-2}} = \sum_{m=0}^{\infty} \frac{ (H_M f(m))^p  (\log (m+3))^{\beta} }{ m+1 } \gtrsim_{\beta,p} \sum_{m=0}^{\infty} \frac{ 1}{ (m+3) \log (m+3)} = \infty. 
\]

\end{proof}

We now turn our attention to the case $\gamma>p-1>0$ and $\gamma-\beta\geq p$. We prove that $H_I\colon  \HL^p_{\gamma,p-2} \to \HL^p_{\beta,p-2}$ is bounded by using the representation $H_M$. 

To this end, let
\[
L^p_{1,\gamma}=\biggl\{ f\colon [e,\infty) \to \mathbb R \text{ measurable} : \norm{f}_{L^p_{1,\gamma}}^p:= \int_e^\infty f(x)^p\frac{(\log x)^{\gamma}}{x}\, dx  \biggr\}
\]
and
\[
L^p_{2,\gamma}=\biggl\{ f\colon [e,\infty) \to \mathbb R \text{ measurable} : \norm{f}_{L^p_{2,\gamma}}^p:= \int_e^\infty f(x)^p\frac{(\log x)^{\gamma-p}}{x}\, dx \biggr\}.
\]

Let $j\geq 0, \, \eta\in\R$. We define the (modified) averaging operator 
\[
A_{j,\eta}f(x) = \chi_{[j,\infty)}(x) x^{-\eta}\int_j^x \frac{1}{t^{1-\eta}} f(t)\, dt.
\]

We also define
\[
B_{j,\eta}f(x) = \chi_{[j,\infty)}(x) x^{\eta-1}\int_x^{\infty} \frac{1}{t^{\eta}} f(t)\, dt,
\]
which is the formal adjoint operator of $A_{j,\eta}$ with respect to the
$L^2(0,\infty)$ pairing.

The Hilbert type operator we will use is given by
\[
H_Lf(x) := \chi_{[e,\infty)}(x)\int_e^\infty\frac{f(y)}{x+y}\, dy.
\]

For completeness we give a simple proof of the following lemma, which for $j=0$ and $\eta=1$ can be found in, for example \cite[327,328]{hardy1934inequalities}. 

\begin{lem}\label{lem:ModAvgBddOnLp}
Let $1<p<\infty$, $j>0$ and $\eta>1-1/p$. The operators $B_{j,\eta}$ and $A_{j,\eta}$ are bounded on $L^p(j,\infty)$ and $L^{p'}(j,\infty)$, respectively.
\end{lem}
\begin{proof}
Let $1/p' +1/p = 1$ and apply Hölder's inequality to obtain
\begin{equation}\label{eq:applHolder}
\int_j^{\infty} \int_j^{\infty} g(x) \bigg(\frac{x}{t}\bigg)^{\frac{1}{p p'}}     f(t)  \bigg(\frac{t}{x}\bigg)^{\frac{1}{p p'}}   \frac{ \chi_{x>t}(x,t) }{t^{1-\eta} x^{\eta}}  \, dx\, dt \lesssim_{\eta,p} \norm{g}_{L^p(j,\infty)} \norm{f}_{L^{p'}(j,\infty)}  
\end{equation}
given that $\eta > 1-1/p$, where $g\in L^p(j,\infty)$ and $f\in L^{p'}(j,\infty)$. Moreover, by taking the supremum over $f$ with unit norm, we have that the left-hand side becomes
\[
\bigg(\int_j^{\infty} \abs{  \int_t^{\infty} g(x)   \frac{t^{\eta-1}}{x^{-\eta}}  \, dx }^p\, dt \bigg)^{\frac{1}{p}} = \norm{B_{j,\eta}g}_{L^p(j,\infty)},
\]
therefore, $B_{j,\eta}\colon L^p(j,\infty) \to L^p(j,\infty)$ is bounded if $\eta>1-1/p$. Taking the supremum in \eqref{eq:applHolder} over $g$ with unit norm, we obtain the statement for $A_{j,\eta}$.
\end{proof}

Define $J_{1,\gamma,p}\colon f\mapsto x^{\gamma/p} f(e^x)$ and $J_{2,\gamma,p}\colon f\mapsto  x^{\gamma/p-1} f(e^x)$. We have the following two lemmas.
\begin{lem}\label{lem:ABbdd}
Let $\gamma >p-1$. The mapping $J_{1,\gamma,p}$ is an isometric isomorphism $L^p_{1,\gamma} \to L^p(1,\infty)$ and $J_{2,\gamma,p}$ is an isometric isomorphism $L^p_{2,\gamma} \to L^p(1,\infty)$. Moreover, $A_{e,1}\colon L^p_{1,\gamma} \to L^p_{2,\gamma}$ and $B_{e,1}\colon L^p_{1,\gamma} \to L^p_{2,\gamma}$ are bounded.
\end{lem}
\begin{proof}
It is clear that for each $k=1,2,
$, the operator $J_{k,\gamma,p}\colon L^p_{k,\gamma} \to L^p(1,\infty)$ is an isometric isomorphism. 
For the second statement, we have for $x\in(1,\infty)$
\begin{equation}\label{eq:AeComparable}
J_{2,\gamma,p} A_{e,1} J_{1,\gamma,p}^{-1} f (x) =  x^{\frac{\gamma}{p}-1} e^{-x}\int_{1}^x \frac{e^t}{t^{\frac{\gamma}{p}}} f(t)\, dt. 
\end{equation}
Furthermore, there exists a $t_0=t_0(\gamma,p)\geq 1$ such that 
\[
t\mapsto \frac{e^t}{t^{\frac{\gamma}{p}}}
\]
is increasing for $t > t_0$. It follows that
\[
J_{2,\gamma,p} A_{e,1} J_{1,\gamma,p}^{-1} f (x) \lesssim A_{t_0,1}f(x)+ x^{\gamma/p-1} e^{-x} \|f\|_{L^p(1,\infty)},
\]
proving that 
$J_{2,\gamma,p} A_{e,1} J_{1,\gamma,p}^{-1}\colon L^p(1,\infty)\to L^p(1,\infty)$ is bounded, and hence, $A_{e,1}\colon L^p_{1,\gamma} \to L^p_{2,\gamma}$ is bounded.

Concerning the boundedness of $B_{e,1}$, we have for $x\in(1,\infty)$

\[
J_{2,\gamma,p} B_{e,1} J_{1,\gamma,p}^{-1} f (x) = x^{\frac{\gamma}{p}-1} \int_x^{\infty} \frac{1}{t^{\frac{\gamma}{p}}} f(t)\, dt \leq   B_{1,\gamma/p} f (x). 
\]
Since $\gamma/p>1-1/p$, Lemma \ref{lem:ModAvgBddOnLp} finishes the proof.
\end{proof}

Before we proceed, we mention that Lemma \ref{lem:ModAvgBddOnLp} could include the parameter $j=0$ if $\eta>\max\{1/p,1-1/p\}$, holds. The reason why the domain of the $L^p$-spaces are $[1,\infty)$ and $[e,\infty)$, and not the more natural $[0,\infty)$ and $[1,\infty)$, is that the right-hand side of \eqref{eq:AeComparable} in Lemma \ref{lem:ABbdd} would diverge.

\begin{prop}\label{prop:HLbdd}
Let $\gamma>p-1>0$. The Hilbert operator $
H_L\colon L^p_{1,\gamma} \to L^p_{2,\gamma}$ is bounded.
\end{prop}
\begin{proof}
The inequality
\[
\frac{1}{x+y} \leq  \chi_{(e,x)}(y) \frac{1}{x} + \chi_{[x,\infty)}(y) \frac{1}{y}, \quad x,y\in[e,\infty)
\]
gives us
\[
H_L f(x) \leq A_{e,1}f(x) + B_{e,1}f(x),
\]
whenever $f\geq 0$. Since the subset of nonnegative functions is a maximizing set for $H_L$, because the kernel $(x,y)\mapsto 1/(x+y),\, x,y>0$ is positive, the statement now follows from Lemma \ref{lem:ABbdd}. 
\end{proof}

For convenience, we define the following variants of the spaces $\HL^p_{\gamma,p-2}$ and $\HL^p_{\beta,p-2}$ for the crucial case when $\gamma-\beta=p$. Let
\[
\ell^p_{1,\gamma}=\biggl\{a=(a_n)_{n=3}^\infty \subset \mathbb R : \norm{a}_{\ell^p_{1,\gamma}}^p:= \sum_{n=3}^\infty \abs{a_n}^p \frac{(\log n)^{\gamma}}{n}\biggr\}
\]
and
\[
\ell^p_{2,\gamma}=\biggl\{a=(a_n)_{n=3}^\infty \subset \mathbb R : \norm{a}_{\ell^p_{2,\gamma}}^p:= \sum_{n=3}^\infty \abs{a_n}^p \frac{(\log n)^{\gamma-p}}{n}\biggr\}.
\]
Furthermore, define the discrete version of $H_L$,
\[
H_l a(m) = \sum_{n=3}^{\infty} \frac{a_n}{m+n}, \  m\in\mathbb N\cap [3,\infty).
\]

Using the subset of $L^p_{1,\gamma}$ consisting of functions constant on $[n,n+1[, \, n\in\mathbb N$, we obtain our first theorem.

\begin{thm}\label{thm:HlBdd}
Let $\gamma>p-1>0$. We have
\[
\norm{H_M}_{\HL^p_{\gamma,p-1} \to \HL^p_{\gamma-p,p-1}} \asymp \norm{H_l}_{\ell^p_{1,\gamma} \to \ell^p_{2,\gamma}}
\lesssim \norm{H_L}_{L^p_1 \to L^p_2} <\infty.
\]
\end{thm}

\begin{proof}
Since the nonnegative sequences/Taylor coefficients are maximizing due to the positive kernels of the operators, the comparability, $\asymp$, follows immediately after an application of shift operators.

To prove the inequality, $\lesssim$, let $a\in\ell_{1,\gamma}^p$ with $a_n\geq 0$ for all $n$, and consider the function $f\colon[e,\infty)\to \mathbb R$ defined by
\[
f_a(x)= \sum_{n=3}^{\infty} a_n  \chi_{[n,n+1)}(x). 
\]
It is clear  $\norm{f_a}_{L_{1,\gamma}^p}^p \asymp \norm{a}_{\ell_{1,\gamma}^p}^p$ and
\begin{align*}
\norm{H_L f_a}_{L_{2,\gamma}^p}^p &= \bigg(\int_e^3 +\sum_{m=3}^{\infty}\int_m^{m+1} \bigg) \bigg[ \bigg(\int_e^3 +\sum_{k=3}^{\infty}\int_k^{k+1} \bigg) \Big( \sum_{n=3}^{\infty} a_n \chi_{[n,n+1)}(y) \Big)  \, \frac{dy}{x+y} \bigg]^p \frac{(\log x)^{\gamma-p}}{x}\, dx\\
&\asymp \bigg( \sum_{k=3}^{\infty} a_k   \, \frac{1}{3+k}  \bigg)^p +\sum_{m=3}^{\infty} \bigg( \sum_{k=3}^{\infty} a_k   \, \frac{1}{m+k}  \bigg)^p \, \frac{(\log m)^{\gamma-p}}{m}\\
&\asymp_{\gamma,p} \sum_{m=3}^{\infty} \bigg( \sum_{k=3}^{\infty} a_k   \, \frac{1}{m+k}  \bigg)^p \, \frac{(\log m)^{\gamma-p}}{m} = \norm{H_l a}_{\ell_{2,\gamma}^p}^p.
\end{align*}
The theorem now follows from Proposition \ref{prop:HLbdd} noting that the nonnegative sequences form a maximizing set of $H_l$ since the kernel is nonnegative.
\end{proof}

Combining Theorem \ref{thm:HlBdd}, Proposition \ref{prop:necessaryCondHL}, and the fact that $\mathbb R\ni\gamma\mapsto \norm{f}_{\HL^p_{\gamma,p-1}}$ is increasing, we obtain Theorem \ref{thm:mainHL}.

We finish the section with the following result.

\begin{cor}\label{cor:OtherOperatorsHL}
Theorem \ref{thm:mainHL} holds if the operator $H_I$ is exchanged with $H_M$ or $\widetilde H_I$.
\end{cor}
\begin{proof}
The statement is clear for $H_M$, since $H_I=H_M$.

Moreover, for a function $f$ with nonnegative Taylor coeffiients, $a_n \geq 0$, $H_I f =\widetilde H_If$, hence, Proposition \ref{prop:necessaryCondHL} gives the statement about unboundedness. 

For a function $f=\sum a_k z^k$, let $f^{||}=\sum \abs{a_k} z^k$. The boundedness statement now follows from
\begin{align*}
\norm{H_I}_{\HL^p_{\gamma,p-2}\to\HL^p_{\beta,p-2}}^p &\leq \norm{\widetilde H_I}_{\HL^p_{\gamma,p-2}\to\HL^p_{\beta,p-2}}^p   \leq  \sup_{\norm{f }_{\HL^p_{\gamma,p-2}} = 1} \norm{\widetilde H_I f^{||}}_{\HL^p_{\beta,p-2}}^p \leq \norm{H_I}_{\HL^p_{\gamma,p-2}\to\HL^p_{\beta,p-2}}^p,
\end{align*}
where the paragraph containing \eqref{eq:CoeffHilbAbs} gives the last inequality, and $\gamma>p-1>0, \, \beta\in\mathbb R$.
\end{proof}

\section{The Hilbert operator between weighted Bergman spaces}

Before we proceed, we present the following elementary lemma, which can be proved by differentiation.
\begin{lem}\label{lem:eventuallyIncreasing}
Let $\gamma_1,\gamma_2\in\mathbb R$. The function
\[
f\colon x\mapsto x \bigg(\log \frac{e^e}{x}\bigg)^{\gamma_1} \bigg(\log \bigg(\log \frac{e^e}{x}\bigg)\bigg)^{\gamma_2},
\]
is bounded on $(0,1)$, and increasing on $(0,\epsilon)$ or some $\epsilon=\epsilon(\gamma_1,\gamma_2)$.
\end{lem}

We start by giving a counterpart to the first part of Proposition \ref{prop:necessaryCondHL}.
\begin{prop}\label{prop:HwellDefined}
The operators $H_I$ and $\widetilde H_I$ are well defined on $A^p_{v_{\gamma,p-2}}$ if and only if $\gamma>p-1$.
\end{prop}
\begin{proof}
It follows from Lemma \ref{lem:MomentsOfAbsH} that $H_I$ and $\widetilde H_I$ are well defined on $A^p_{v_{\gamma,p-2}}$ if $\gamma>p-1$.

We proceed to assume $\gamma\leq p-1$, and because $A^p_{v_{p-1,p-2}} \subset A^p_{v_{\gamma,p-2}}$, it suffices to consider $\gamma=p-1$. Consider the function

\[
f(z) = \frac{1}{(1-z) \big(\log \frac{e^e}{1-z}\big)\big(\log\big(\log \frac{e^e}{1-z}\big)\big)}.
\]

It is clear that 
\[
H_I f(0) = \widetilde H_I f(0) =  \int_0^1 f(t) \, dt= \int_{e^e}^\infty \frac{dt}{ t (\log t) (\log(\log t) )  }  = \lim_{t\to +\infty}\log(\log(\log t)) = \infty.
\]

Moreover, to see that $f\in A^p_{v_{p-1,p-2}}$, we proceed by first noting that $t^2/(1-\cos t) \in (2,5)$ when $t\in (-\pi,\pi)\setminus \{0\}$ yielding 
\[
\abs{1-re^{it}}^2 = (1-r)^2+2r(1-\cos t) \leq (1-r+\abs{t})^2 \leq 10 \abs{1-re^{it}}^2 ,\quad t\in (-\pi,\pi), \, r\in (0,1).
\]
Using $\abs{\log z} \geq \Re \log z = \log \abs{z},\, z\neq 0 $, and Lemma \ref{lem:eventuallyIncreasing}, we have for $r\in (0,1)$
\begin{align*}
\int_{-\pi}^{\pi}  \abs{  f(re^{i\theta})    }^p \, d\theta & \leq \int_{-\pi}^{\pi}   \frac{1}{\abs{1-re^{i\theta}}^p \big(\log \frac{e^e}{\abs{1-re^{i\theta}}}\big)^p\big(\log\big(\log \frac{e^e}{\abs{1-re^{i\theta}}}\big)\big)^p} \, d\theta \\
&\lesssim_p \int_0^{\pi}   \frac{1}{(1-r+t)^p \big(\log \frac{e^e}{1-r+t}\big)^p\big(\log\big(\log \frac{e^e}{1-r+t}\big)\big)^p} \, dt \\
&\leq \int_{1-r}^{\pi + 1}   \frac{    t^{-\frac{p-1}{2}}   }{ \big(\log \frac{e^e}{t}\big)^p\big(\log\big(\log \frac{e^e}{t}\big)\big)^p} \frac{1}{t^{\frac{p+1}{2}}} \, dt  \\
&\lesssim_p  \frac{    (1-r)^{-\frac{p-1}{2}}   }{ \big(\log \frac{e^e}{1-r}\big)^p\big(\log\big(\log \frac{e^e}{1-r}\big)\big)^p} \int_{1-r}^1   \frac{1}{t^{\frac{p+1}{2}}} \, dt + M(p) \\
&\lesssim_p    \frac{    (1-r)^{1-p}   }{ \big(\log \frac{e^e}{1-r}\big)^p\big(\log\big(\log \frac{e^e}{1-r}\big)\big)^p}   
\end{align*}
for some constant $M(p)>0$. It follows that
\begin{align*}
\norm{ f  }_{A^p_{v_{p-1,p-2}}}^p &\lesssim_p \int_0^1   \frac{    (1-r)^{1-p}   }{ \big(\log \frac{e^e}{1-r}\big)^p\big(\log\big(\log \frac{e^e}{1-r}\big)\big)^p}  \log(\frac{e}{1-r} )^{p-1}  (1-r)^{p-2} \, dr \\
&\leq  \int_0^1   \frac{   1  }{(1-r) \big(\log \frac{e^e}{1-r}\big)\big(\log\big(\log \frac{e^e}{1-r}\big)\big)^p}   \, dr,
\end{align*}
which is finite since $p>1$.
\end{proof}

Now, we present a weaker version of the second part of Proposition \ref{prop:necessaryCondHL} for $A^p_{v_{p-1,p-2}},\, p>1$.

Define the classical sequence of functions \((f_{n}(z))_{n=0}^{\infty} = ((1-z)^{-c_n})_{n=0}^{\infty}\), where \(c_n = (\alpha+2)/p - 1/(n+1), n\in\mathbb N\). In our case $p=\alpha+2$, so the expression simplifies. To prove unboundedness, we utilize the classical representation of $H_I$, introduced in \cite[(2.2)]{Diamantopoulos2000}, and used on $A^p_{\alpha}, \, p>2+\alpha> -1$ in \cite[p.~519]{Karapetrovic2018}. We have
\begin{equation}\label{eq:UsefulRepresentationOfH}
H_If(z) = \int_0^1 \frac{\phi_t(z) }{t} f(\phi_t(z)) \, dt ,
\end{equation}
where
\[
\phi_t(z) = \frac{t}{1-(1-t)z}, \quad t\in (0,1),\ z\in \D.
\]
For verification that this representation is valid on $A^p_{v_{\gamma,p-2}},\, \gamma>p-1$, we refer to \cite[pp.~4--6]{Lindstrom2022} and Lemma \ref{lem:MomentsOfAbsH}, which proves that 
\[
A^p_{v_{\gamma,p-2}} \subset A^1_{[0,1)} := \Big\{ f\in\HOLO(\D) :  \int_0^1 \abs{f(r)} \, dr <\infty \Big\}.
\]

\begin{prop}\label{prop:HIonBergmanNotBdd}
Let $\gamma >p-1>0$. If $\gamma-\beta<p$, then $H_I\colon A^p_{v_{\gamma,p-2}} \to A^p_{v_{\beta,p-2}}$ is not bounded.
\end{prop}

\begin{proof}
Using \cite[Theorem 1.12]{zhu2005spaces}, we have for $n\in\mathbb N$

\begin{align*}
\norm{f_{n}}_{A^p_{v_{\gamma,p-2}}}^p &= \int_{\D} \abs{ 1-z }^{-p+p/(n+1)}  v_{\gamma,p-2}(z) \, dA(z)\  \asymp  \int_0^1 ( 1-r )^{1-p+p/(n+1)}  v_{\gamma,p-2}(r) \, dr \\
&\asymp_{p,\gamma}  \int_0^1 ( 1-r )^{-1+p/(n+1)}   \bigg(\log\frac{e}{1-r} \bigg)^{\gamma}  \, dr \\
&= \int_0^\infty ( e^{-r})^{p/(n+1)}   (1+r)^{\gamma}  \, dr\\
&= \frac{ (n+1)^{\gamma+1}  \int_0^\infty e^{-r} \Big(\frac{p}{n+1}  +r\big)^{\gamma}  \, dr  }{  p^{\gamma+1}   },
\end{align*}
which shows that $f_{n}\in A^p_{v_{\gamma,p-2}}$ for every $\gamma\in\mathbb R$ and $n\in\mathbb N$. Moreover, for $\gamma>-1$ it holds that

\begin{equation}\label{eq:asympOfTestFunctions}
\norm{f_{n}}_{A^p_{v_{\gamma,p-2}}}^p \asymp_{\gamma,p} (n+1)^{\gamma+1}.
\end{equation}

To proceed, representation \eqref{eq:UsefulRepresentationOfH} allows us to treat $H_I $ as a multiplication operator on the family $\{f_n:n\in\mathbb N\}$. Indeed,
\[
H_I f_n = f_n  C_n,
\]
where
\[
C_n(z) = \int_0^1 \frac{  (1-(1-t)z)^{c_n-1}    }{  (1-t)^{c_n}  }  \, dt,\quad z\in \D.
\]
Since $c_n = 1- 1/(n+1) \in [1/2,1)$, we have for $z\in\D$
\[
\frac{1}{2}\int_0^1  \frac{ t^{\frac{1}{2}} }{ (1-t)^{c_n} } \, dt  \leq  \int_0^1 \frac{\Re (1-(1-t)\CONJ{z})^{1-c_n}}{ \abs{ 1-(1-t)z }^{2(1-c_n)}  } \frac{1}{(1-t)^{c_n}} \, dt = \Re C_n(z)\leq \abs{C_n(z)} \leq  \int_0^1  \frac{ t^{-\frac{1}{2}} }{ (1-t)^{c_n} } \, dt,
\]
which yields $\abs{C_n(z)}  \asymp (1-c_n)^{-1} = n+1$.
Since $\beta>-1$, it now follows from \eqref{eq:asympOfTestFunctions} that  
\[
\norm{H_I f_n}_{A^p_{v_{\beta,p-2}}} \asymp (1-c_n)^{-1} \norm{f_n}_{A^p_{v_{\beta,p-2}}} \asymp_{p,\beta} (n+1)^{1+(\beta+1)/p} , \quad n\in \mathbb N.   
\]
Now \eqref{eq:asympOfTestFunctions} yields that 
\[
\lim_{n\to \infty} \frac{ \norm{H_I f_n}_{A^p_{v_{\beta,p-2}}} }{  \norm{f_n}_{A^p_{v_{\gamma,p-2}}} } = \infty 
\]
when $\beta>\gamma-p$.

\end{proof}

Concerning a sufficient condition for boundedness, we have the following proposition. The strategy is essentially outlined in the proof of \cite[Lemma 8]{Galanopoulos2014}, but some non trivial modifications are needed due to the logarithmic weights. 

\begin{prop}\label{prop:HIbddNonSharp}
Let $\gamma >p-1>0$. If $\gamma-\beta>p$, then both $H_I$ and $\widetilde H_I$ act boundedly $ A^p_{v_{\gamma,p-2}} \to A^p_{v_{\beta,p-2}}$.
\end{prop}

In the proof we use the common abbreviation
\begin{equation}\label{eq:DeltaDefinition}
\Delta_0 f = a_0 \ \text{ and } \ \Delta_n f = \sum_{k=2^{n-1}} ^{2^n-1} a_k z^k \text{ where }  f\colon z\mapsto \sum_{n=0}^{\infty} a_n z^n \, z\in \mathbb D, \text{ and } n\in \mathbb N.
\end{equation}

\begin{proof}
Let $f\in A^p_{v_{\gamma,p-2}}$ with $\norm{f}_{A^p_{v_{\gamma,p-2}}} = 1$. For $r\in (0,1]$, put $\psi(r) = r v_{\beta,p-2}(\sqrt{1-r})$ and let $\delta=\sup\{0<\delta'<1 : \psi|_{(0,\delta')} \text{ is increasing}\}$. Notice that $\delta>0$ by Lemma \ref{lem:eventuallyIncreasing}, and define 
\[
\phi(r) = \begin{cases}
\psi(r)^{\frac{1}{p}}, & 0<r \leq \delta;\\  
\norm{\psi}_{L^{\infty}(0,1)}^{\frac{1}{p}}, & r>\delta.
\end{cases}
\]
Then we can apply \cite[Theorem 2.1(b)]{Mateljevic1984} with $q=p$ and since $\phi(1-r)^p/(1-r) \asymp v_{\beta,p-2}(r), \, r\in(0,1)$, we have 
\begin{equation}\label{eq:formulaForNorm}
\begin{split}
\norm{H_I f}_{A^p_{v_{\beta,p-2}}}^p &\asymp_{\beta,p}  \sum_{n=0}^{\infty} \phi(2^{-n})^p \int_0^{2\pi} \abs{ \Delta_n H_I f(e^{i\theta})  }^p \, d\theta \asymp_{\beta,p} \sum_{n=0}^{\infty} 2^{-n(p-1)} (n+1)^\beta \int_0^{2\pi} \abs{ \Delta_n H_I f(e^{i\theta})  }^p \, d\theta \\
& = 2\pi \abs{\int_0^1 f(t) \, dt    }^p + \sum_{n=1}^{\infty} 2^{-n(p-1)} (n+1)^\beta \int_0^{2\pi} \abs{ \Delta_n H_I f(e^{i\theta})  }^p \, d\theta.
\end{split}
\end{equation}
Moreover, by Minkowski's inequality, Lemma A on page~626 in \cite{Pavlovic2013}, and Lemma \ref{lem:MomentsOfAbsH}, we have for $n\in\mathbb N$ 
\begin{equation}\label{eq:MomentsOfHI}
\begin{split}
\bigg( \int_0^{2\pi} \abs{ \Delta_n H_I f(e^{i\theta})  }^p \, d\theta \bigg)^{\frac{1}{p}} &= \bigg( \int_0^{2\pi} \abs{   \sum_{k=2^{n-1}}^{2^n-1} e^{ik\theta} \int_0^1 t^k f(t) \, dt }^p \, d\theta  \bigg)^{\frac{1}{p}} \\
&\leq \int_0^1 \abs{f(t)} \bigg( \int_0^{2\pi} \abs{   \sum_{k=2^{n-1}} ^{2^n-1} e^{ik\theta}  t^k   }^p \, d\theta  \bigg)^{\frac{1}{p}} \, dt \\
&\lesssim_p \int_0^1 \abs{f(t)} t^{2^{n-1}} \bigg( \int_0^{2\pi} \abs{   \sum_{k=2^{n-1}} ^{2^n-1} e^{ik\theta}     }^p \, d\theta  \bigg)^{\frac{1}{p}} \, dt \\
&=\int_0^1 \abs{f(t)} t^{2^{n-1}} \, dt \, \bigg( \int_0^{2\pi} \abs{ \frac{\sin (2^{n-2} \theta)}{ \sin (\theta/2)}     }^p \, d\theta  \bigg)^{\frac{1}{p}}  \\
&\asymp_p \int_0^1 \abs{f(t)} t^{2^{n-1}} \, dt \, 2^{n(1-1/p)}\\
&\lesssim_p 
\bigg(\int_0^1  \bigg(\frac{ \log(\frac{e}{1-t} )^{-\gamma}  }{(1-t)^{p-1}}\bigg)^{\frac{p'}{p}} t^{2^{n-1}} \, dt   \bigg)^{\frac{1}{p'}} \norm{f}_{A^p_{v_{\gamma,\alpha}}} 2^{n(1-1/p)}.
\end{split}
\end{equation}

Combining this with \eqref{eq:formulaForNorm} and Lemmas \ref{lem:asympOfWeightInSeries} and \ref{lem:MomentsOfAbsH}, we obtain when $\gamma> p-1$,
\begin{equation}\label{eq:HIbddFinal}
\begin{split}
\norm{H_I f}_{A^p_{v_{\beta,p-2}}}^p &\lesssim_{\beta,p}  \bigg(\int_0^1\abs{f(t)} \, dt\bigg)^p + \sum_{n=1}^{\infty} (n+1)^\beta \bigg( \int_0^1 \abs{f(t)} t^{2^{n-1}} \, dt \bigg)^p \\
&\lesssim_{\gamma,p} \norm{f}_{A^p_{v_{\gamma,p-2}}} \bigg(1+\sum_{n=1}^{\infty} (n+1)^\beta  (\log(2^{n-1}+3))^{p-1-\gamma}\bigg)\\
&\asymp_{\gamma,p}\norm{f}_{A^p_{v_{\gamma,p-2}}} \sum_{n=0}^{\infty} (n+1)^{\beta-\gamma+p-1},
\end{split}
\end{equation}
where the right hand side is finite since $\gamma-\beta>p$. If we replace $H_I$ by $\widetilde H_I$, the same proof works without other modifications than replacing $f(t)$ by $\abs{f(t)}$ in two places (one in \eqref{eq:formulaForNorm} and one in \eqref{eq:MomentsOfHI}). 
\end{proof}

A comparison with the generalized Hardy-Littlewood spaces gives us the following proposition, which together with Proposition \ref{prop:HIonBergmanNotBdd} and Proposition \ref{prop:HIbddNonSharp} proves Theorem \ref{thm:main}.
\begin{prop}\label{prop:BetterResultForPleq2}
Let $1<p\leq 2$ and $\gamma >p-1$. If $\gamma-\beta \geq p$, then both $H_I$ and $\widetilde H_I$ act boundedly $A^p_{v_{\gamma,p-2}} \to A^p_{v_{\beta,p-2}}$.
\end{prop}

A part of the proof of the proposition is given in the following lemma, which can be compared with \cite[Lemma 8 (ii)]{Galanopoulos2014}.

\begin{lem}\label{lem:usefulRepresentationHIAp}
Let $\gamma>p-1>0$ and $\beta \leq \gamma-p$. Then
\[
\norm{\widetilde H_I f}_{A^p_{v_{\beta,p-2}}}^p \asymp_{\beta,p} \bigg(\int_0^1 \abs{f(t)} \, dt\bigg)^p + \sum_{n=0}^{\infty}  (n+1)^\beta \bigg(\int_0^1 t^{2^{n-1}} \abs{f(t)} \, dt \bigg)^p.
\]
\end{lem}

\begin{proof}
By \cite[Lemma A]{Pavlovic2013}, we have for $0<r<1$, $n\in\mathbb N$ and $f\in \HOLO(\D)$
\begin{equation}
\begin{split}
\int_0^1 t^{2^n} \abs{f(t)} \, dt \,  \bigg(  \int_0^{2\pi} \abs{   \sum_{k=2^{n-1}} ^{2^n-1} e^{ik\theta}  }^p \, d\theta  \bigg)^{\frac{1}{p}} &\lesssim_p \bigg( \int_0^{2\pi} \abs{   \sum_{k=2^{n-1}} ^{2^n-1} e^{ik\theta} \int_0^1 t^k \abs{f(t)} \, dt }^p \, d\theta  \bigg)^{\frac{1}{p}} \\
&\lesssim_p \int_0^1 t^{2^{n-1}} \abs{f(t)} \, dt \, \bigg(   \int_0^{2\pi} \abs{   \sum_{k=2^{n-1}} ^{2^n-1} e^{ik\theta} }^p \, d\theta  \bigg)^{\frac{1}{p}}.  
\end{split}
\end{equation}
Similarly to \eqref{eq:MomentsOfHI}, it follows that for $n\in\mathbb N$
\begin{align*}
2^{n(1-1/p)} \int_0^1 t^{2^n} \abs{f(t)} \, dt &\lesssim_p \bigg( \int_0^{2\pi} \abs{ \Delta_n \widetilde H_I f(e^{i\theta})  }^p \, d\theta \bigg)^{\frac{1}{p}} \\
&\lesssim_p 2^{n(1-1/p)}\int_0^1 t^{2^{n-1}} \abs{f(t)} \, dt,  
\end{align*}
and by calculations comparable to \eqref{eq:formulaForNorm}, we can conclude that
\[
\norm{\widetilde H_I f}_{A^p_{v_{\beta,p-2}}}^p
\lesssim_{\beta,p} \bigg(\int_0^1 \abs{f(t)} \, dt\bigg)^p + \sum_{n=0}^{\infty}  (n+1)^\beta \bigg(\int_0^1 t^{2^{n-1}} \abs{f(t)} \, dt\bigg)^p
\]
and
\[
\norm{\widetilde H_I f}_{A^p_{v_{\beta,p-2}}}^p
\gtrsim_{\beta,p} \bigg(\int_0^1 \abs{f(t)} \, dt\bigg)^p + \sum_{n=1}^{\infty}  n^\beta  \bigg(\int_0^1 t^{2^{n-1}} \abs{f(t)} \, dt\bigg)^p, 
\]

which yields
\[
\norm{\widetilde H_I f}_{A^p_{v_{\beta,p-2}}}^p \asymp_{\beta,p} \bigg(\int_0^1 \abs{f(t)} \, dt\bigg)^p + \sum_{n=0}^{\infty}  (n+1)^\beta \bigg(\int_0^1 t^{2^{n-1}} \abs{f(t)} \, dt \bigg)^p.
\]
\end{proof}

\begin{proof}[Proof of Proposition \ref{prop:BetterResultForPleq2}]
From Lemma \ref{lem:HLineq} and Corollary \ref{cor:OtherOperatorsHL}, it follows that
\begin{equation}\label{eq:tildeHIbdd}
\norm{\widetilde H_I f}_{ A^p_{v_{\beta,p-2}}} \asymp_{\beta,p} \norm{\widetilde H_I f}_{ \HL^p_{\beta,p-2}} \lesssim_{\beta,\gamma,p} \norm{f}_{\HL^p_{\gamma,p-2}}\lesssim_{\gamma,p} \norm{f}_{ A^p_{v_{\gamma,p-2}}}. 
\end{equation}
Since all calculations in Proposition \ref{prop:HIbddNonSharp}, except the final conclusion, hold when $\gamma-\beta\geq p$, we can apply the first $\lesssim$ in \eqref{eq:HIbddFinal} and Lemma \ref{lem:usefulRepresentationHIAp} followed by \eqref{eq:tildeHIbdd} to obtain
\[
\norm{H_I f}_{A^p_{v_{\beta,p-2}}}^p \lesssim_{\beta,p}   \norm{\widetilde H_I f}_{A^p_{v_{\beta,p-2}}}^p \lesssim_{\gamma,p} \norm{f}_{ A^p_{v_{\gamma,p-2}}}.
\]

\end{proof}

Proposition \ref{prop:BetterResultForPleq2}

By straightforward modifications to the proofs of

We finish the paper by giving a final corollary concerning the generalized Hilbert operator $H_g$ introduced in \cite{Galanopoulos2014}. To this end, we present the following definition of the analytic mean Lipschitz spaces, where $\Delta_n$ is defined in \eqref{eq:DeltaDefinition}: 
\[
\Lambda(p,\alpha) = \bigg\{ f\in H^p : \norm{f}_{\Lambda(p,\alpha)}^p = \abs{f(0)}^p + \sup_{n\in\mathbb N\cup\{0\}} 2^{n(\alpha-1)p} \int_0^{2\pi} \abs{ \Delta_n(f')(re^{it}) }^p \frac{dt}{2\pi}\bigg\},\quad p>1,\, 0<\alpha< 1;
\]
see \cite{Galanopoulos2014} for some equivalent norms.

We also define the bounded mean dyadic space 
\[
\BDM(p)=\bigg\{ \sum_{n=0}^{\infty} a_n z^n\in \HOLO(D) : \norm{f}_{\BDM(p)}^p = \abs{a_0}^p +  \sup_{n\in \mathbb N} 2^{1-n} \sum_{k=2^{n-1}}^{2^n-1} \abs{a_k}^p  \bigg\}.
\]

The operators of interest are 
\[
H_g(f)(z) = \int_0^1 f(t)g'(tz) \, dt \quad\text{ and }\quad \widetilde H_g(f)(z) = \int_0^1 \abs{f(t)} g'(tz) \, dt \quad z\in\mathbb \D,\, g\in \HOLO(\D).
\]

Our final result concerns the boundedness of these generalized Hilbert operators.

\begin{thm}\label{thm:mainHg}
Let $\gamma>p-1>0$. Then 
\begin{enumerate}[label={(\roman*)}]
\item\label{item:HgHL}
\[
\sup_{\norm{g'}_{\BDM(p)}=1}\norm{H_g}_{\HL^p_{\gamma,p-2} \to \HL^p_{\beta,p-2}} <\infty
\]
if and only if $\gamma-\beta\geq p$. 
\end{enumerate}

Similarly, 
\begin{enumerate}[resume,label={(\roman*)}]
\item\label{item:pleq2Hg}
for $p\leq 2$
\[
\sup_{\norm{g}_{\Lambda(p,1/p)}=1}\norm{H_g}_{ A^p_{v_{\gamma,p-2}} \to A^p_{v_{\beta,p-2}}} <\infty
\]
if and only if $\gamma-\beta\geq p$, and 
\item\label{item:pgeq2Hg} for $p>2$
\[
\sup_{\norm{g}_{\Lambda(p,1/p)}=1}\norm{H_g}_{ A^p_{v_{\gamma,p-2}} \to A^p_{v_{\beta,p-2}}} <\infty
\]
if $\gamma-\beta > p$ and the supremum is infinite if $\gamma-\beta < p$.
\end{enumerate}

The sufficient conditions for the suprema to be finite are still sufficient if $H_g$ is replaced by $\widetilde H_g$. The equivalence in \ref{item:HgHL} is also true considering $\widetilde H_g$.

\end{thm}
\begin{proof}
Consider the function $g\colon z \mapsto 1-\log(1-z)$ for which $g'\colon z\mapsto (1-z)^{-1}$. It is clear that $g'\in \BDM(p)$ and some of the calculations in \eqref{eq:MomentsOfHI} show that $g\in \Lambda(p,1/p)$. It follows from Theorems \ref{thm:main} and \ref{thm:mainHL} that the suprema are infinite if $\gamma-\beta<p$. Assume now that $\gamma-\beta\geq p$. We prove that 
\begin{equation}\label{eq:HgHL}
\norm{H_g}_{ \HL^p_{\gamma,p-2} \to \HL^p_{\beta,p-2}}\lesssim_{\beta,p} \norm{\widetilde H_I}_{ \HL^p_{\gamma,p-2} \to \HL^p_{\beta,p-2}} \norm{g'}_{\BDM(p)} , \quad g'\in\BDM(p) 
\end{equation}
and
\begin{equation}\label{eq:HgAp}
\norm{H_g}_{ A^p_{v_{\gamma,p-2}} \to A^p_{v_{\beta,p-2}}}
\lesssim_{\gamma,\beta,p}
 \norm{\widetilde H_I}_{ A^p_{v_{\gamma,p-2}} \to A^p_{v_{\beta,p-2}}} \norm{g}_{\Lambda(p,1/p)}, \quad g\in\Lambda(p,1/p).
\end{equation}

To this end, let $(c_k)$ be the Taylor coefficients of $g'$. We have for $f\in \HL^p_{\gamma,p-2}$
\begin{align*}
\norm{H_g f}_{\HL^p_{\beta,p-2}}^p &= \sum_{n=0}^{\infty}  \frac{(\log(n+3))^{\beta}}{n+3}  \abs{\int_0^1 f(t) t^n \, dt  }^p\abs{c_n}^p \\
&\leq \frac{(\log 3)^{\beta}}{3}  \bigg(\int_0^1 \abs{f(t)} \, dt  \bigg)^p\abs{c_0}^p +\sum_{k=1}^{\infty} \sum_{n=2^{k-1}}^{2^k-1}  \frac{(\log(n+3))^{\beta}}{n+3}  \bigg(\int_0^1 \abs{f(t)} t^n \, dt \bigg)^p\abs{c_n}^p\\
&\lesssim_{\beta}   \bigg(\int_0^1 \abs{f(t)} \, dt  \bigg)^p\abs{c_0}^p +  \sum_{k=1}^{\infty} \frac{k^{\beta}}{2^{k-1}}  \bigg(\int_0^1 \abs{f(t)} t^{2^{k-1}} \, dt \bigg)^p  \sum_{n=2^{k-1}}^{2^k-1}  \abs{c_n}^p \\
&\leq  \bigg[ \bigg(\int_0^1 \abs{f(t)} \, dt  \bigg)^p +  \sum_{k=1}^{\infty} k^{\beta}  \bigg(\int_0^1 \abs{f(t)} t^{2^{k-1}} \, dt \bigg)^p  \bigg] \norm{g'}_{\BDM(p)}^p,
\end{align*}
which combined with Lemmas \ref{lem:usefulRepresentationHIAp} and \ref{lem:HLineq} gives \eqref{eq:HgHL}. Corollary \ref{cor:OtherOperatorsHL} now shows that \ref{item:HgHL} holds.

Similarly to \eqref{eq:formulaForNorm} and \eqref{eq:MomentsOfHI}, we have for $f\in A^p_{v_{\gamma,p-2}}$
\begin{align*}
\norm{H_g f}_{ A^p_{v_{\beta,p-2}}}^p &\asymp_{\beta,p} \sum_{n=0}^{\infty} 2^{-n(p-1)} (n+1)^\beta \int_0^{2\pi} \abs{ \Delta_n H_g f(e^{i\theta})  }^p \, d\theta \\
&= 2\pi \abs{c_0 \int_0^1 f(t) \, dt }^p+ \sum_{n=1}^{\infty} 2^{-n(p-1)} (n+1)^\beta \int_0^{2\pi} \abs{   \sum_{k=2^{n-1}}^{2^n-1} c_k e^{ik\theta} \int_0^1 t^k f(t) \, dt }^p \, d\theta  \\
&\lesssim_p \abs{c_0}^p \bigg(\int_0^1 \abs{f(t)} \, dt \bigg)^p+\sum_{n=0}^{\infty}  \frac{ (n+1)^\beta }{ 2^{n(p-1)}  } 
\bigg(\int_0^1 \abs{f(t)} t^{2^{n-1}} \, dt\bigg)^p \,   \int_0^{2\pi} \abs{   \sum_{k=2^{n-1}} ^{2^n-1} c_k e^{ik\theta}     }^p \, d\theta\\  
&\lesssim \bigg[ \bigg(\int_0^1 \abs{f(t)} \, dt \bigg)^p+\sum_{n=1}^{\infty} (n+1)^\beta 
\bigg(\int_0^1 \abs{f(t)} t^{2^{n-1}} \, dt\bigg)^p \bigg] \norm{g}_{\Lambda(p,1/p)}^p\\
&\asymp_{\beta,p} \norm{\widetilde H_I f}_{A^p_{v_{\beta,p-2}}}^p \norm{g}_{\Lambda(p,1/p)}^p, 
\end{align*}
where the last step follows from Lemma \ref{lem:usefulRepresentationHIAp}. The propositions \ref{prop:BetterResultForPleq2} and \ref{prop:HIbddNonSharp} now give  \ref{item:pleq2Hg} and \ref{item:pgeq2Hg}, respectively.

The sufficiency statements about $\widetilde H_g$ are proved similarly to the above; the necessity statement follows from Corollary \ref{cor:OtherOperatorsHL}.
\end{proof}

\section{Acknowledgments}
The author was funded by the Ruth and Nils-Erik Stenbäck Foundation and the work has been carried out at NTNU -- Norwegian University of Science and Technology.

\bibliographystyle{abbrv}
\bibliography{bibliography}

\end{document}